\documentclass[11pt]{amsart}
\usepackage{amscd, amsmath, amsthm, amssymb}
\usepackage{color}
\usepackage{graphicx}
\usepackage{mathrsfs}
\usepackage{enumerate}
\usepackage[all]{xy}

\newtheorem{theorem}{Theorem}[section]
\newtheorem{lemma}[theorem]{Lemma}
\newtheorem{proposition}[theorem]{Proposition}
\newtheorem{conjecture}[theorem]{Conjecture}
\newtheorem{corollary}[theorem]{Corollary}
\newtheorem{assumption}[theorem]{Assumption}

\theoremstyle{definition}
\newtheorem{definition}[theorem]{Definition}
\newtheorem{notation}[theorem]{Notation}

\newtheorem{construction}[theorem]{Construction}

\theoremstyle{remark}
\newtheorem{remark}[theorem]{Remark}
\newtheorem{claim}[theorem]{Claim}

\numberwithin{equation}{section}

\newcommand\AAA{\mathcal{A}} \newcommand\AAc{\mathcal{A}^{\circ}} \newcommand\CCp{\mathcal{C}^{+}}
\newcommand\kba{\overline{K}} 
\newcommand\NN{{\mathcal{N}}}   \newcommand\PP{{\mathbb{P}}}  \newcommand\spp{{\mathrm{sp}}}    

\newcommand\tx{{\rm \widetilde{X}}}
\newcommand\wnod{W^{\mathrm{nod}}}
\newcommand\GL{{\rm GL}} \newcommand\NS{{\rm NS}} \newcommand\Bir{{\rm Bir}} \newcommand\Aut{{\rm Aut}} \newcommand\Amp{{\rm Amp}}  \newcommand\Nef{{\rm Nef}} \newcommand\Pos{{\rm Pos}} \newcommand\Ker{{\rm Ker}} \newcommand\Spec{{\rm Spec}}
\newcommand\Nod{{\mathrm{Nod}}} \newcommand\Br{\text{\rm Br}}  
 \newcommand\tfx{\widetilde{\mathfrak{X}}}

\begin{document}

\title[automorphisms and the cone conjecture]{On automorphisms and the cone conjecture for Enriques surfaces in odd characteristic}

\author[Long\ \ Wang]{Long\ \ Wang}
\address{Graduate School of Mathematical Sciences, the University of Tokyo, 3-8-1 Komaba, Meguro-Ku, Tokyo 153-8914, Japan}
\email{wangl11@ms.u-tokyo.ac.jp}

\date{\today}

\begin{abstract} We prove that, for an Enriques surface in odd characteristic, the automorphism group is finitely generated and it acts on the effective nef cone with a rational polyhedral fundamental domain. We also construct a smooth projective surface in odd characteristic which is birational to an Enriques surface and whose automorphism group is discrete but not finitely generated.
\end{abstract}

\maketitle


\setcounter{section}{0}
\section{Introduction}

Morrison \cite{Mo93} and Kawamata \cite{Ka97} proposed the following conjecture concerning the structure of the effective nef cone of a \textit{Calabi--Yau variety} (i.e., a normal projective variety with mild singularities whose canonical divisor is numerically trivial).

\begin{conjecture}[Cone conjecture] Let $X$ be a Calabi--Yau variety over an algebraically closed field. Then there exists a rational polyhedral cone $\Pi$ which is a fundamental domain for the action of the automorphism group $\Aut(X)$ on the effective nef cone $\Nef^e(X) := \Nef(X)\cap \mathrm{Eff}(X)$, in the sense that $\Nef^e(X) =\bigcup_{g\in\Aut(X)}g^{\ast}\Pi$, and $\Pi^{\circ} \cap (g^{\ast}\Pi)^{\circ} = \varnothing$ unless $g^{\ast} = \mathrm{id}$.
\end{conjecture}

Over the field of complex numbers, the conjecture was verified in the $2$-dimensional case (see \cite{To10}). However, even for smooth Calabi--Yau threefolds, the conjecture is widely open. We refer to \cite{LOP18} for a survey.

In this paper, we concentrate on the case of Enriques surfaces over a field of odd characteristic. The restriction on the characteristic is used to ensure that these surfaces admit an \'{e}tale cover by a K3 surface. Recently Lieblich--Maulik \cite{LM18} proved the cone conjecture for K3 surfaces in odd characteristic. As the study of Enriques surfaces is often reduced to their universal covers which are K3 surfaces, it is natural to verify the cone conjecture for Enriques surfaces using the corresponding result on K3 surfaces. See \cite{Na85, Ka97, OS01} for the corresponding result over the complex numbers. More precisely, we prove

\begin{theorem}\label{cone} Let $X$ be an Enriques surface over an algebraically closed field $k$ of odd characteristic. Then the action of the automorphism group $\Aut(X)$ on the effective nef cone $\Nef^{e}(X)$ has a rational polyhedral fundamental domain.
\end{theorem}

As a corollary of the cone theorem, we show the finiteness of smooth rational curves and elliptic fibrations up to automorphisms on an Enriques surface (see Corollary \ref{smoothrational} and \ref{ellipticpencil}). The corresponding result over complex number field was proved in \cite{Na85}.

The study of the automorphism groups of Calabi--Yau varieties is not only interrelated with the cone conjecture but also of independent interest. We refer to \cite{Do16} for a survey on the automorphism group of an Enriques surface. We prove here the following finiteness result, which is due to Dolgachev \cite{Do84} over the complex numbers (see also Corollary \ref{do3.45}).

\begin{theorem}\label{finiteness} Let $X$ be an Enriques surface over an algebraically closed field $k$ of odd characteristic. Then the automorphism group $\Aut(X)$ of $X$ is finitely generated.
\end{theorem}

In the rest of this paper, we construct a smooth projective surface in odd characteristic with discrete and non-finitely generated automorphism group by blowing up a certain Enriques surface.

\begin{theorem}\label{nonfinite} Let $p > 3$ be a prime. Then for any algebraically closed field $\mathbf{K}$ containing $\mathbb{F}_p(t)$, there is a smooth projective surface $Y$ birational to some Enriques surface such that $\Aut (Y/\mathbf{K})$ is discrete but not finitely generated.
\end{theorem}

Chronologically, Lesieutre \cite{Le18} gave the first example of a $6$-dimensional smooth projective variety whose automorphism group is discrete and non-finitely generated over a field of arbitrary characteristic. Later Dinh--Oguiso \cite{DO19} constructed a smooth complex projective surface with discrete and non-finitely generated automorphism group by blowing up a certain K3 surface. Oguiso \cite{Og19} then extended their work to odd characteristic. Very recently, Keum--Oguiso \cite{KO19} obtained new examples by blowing up some complex Enriques surface. We show that the same construction as in \cite{KO19} gives the desired example, which answers a question of \cite{KO19}.

\begin{remark} Replacing ``fibration'' by ``quasi-fibration'' as in \cite{Og19}, one should obtain the desired example by the same construction for $p = 3$. For simplicity, we omit this case here.
\end{remark}

The following result as a supplement to Theorem \ref{nonfinite} might be of independent interest (see \cite{Og19} for the case of K3 surfaces).

\begin{proposition}\label{finitefield} Let $p$ be an odd prime integer. Then for any smooth projective surface $Y$ birational to an Enriques surface over $\overline{\mathbb{F}}_p$ and for any field extension $\mathbf{L}/\overline{\mathbb{F}}_p$, the automorphism group $\Aut(Y_{\mathbf{L}}/\mathbf{L})$ is finitely generated.
\end{proposition}

After some preliminaries in Section \ref{sec2}, we prove Theorem \ref{cone} and \ref{finiteness} in Section \ref{sec3}, following the argument in \cite[Section 1]{OS01}. We pass to the K3-cover of an Enriques surface. For supersingular K3-covers, we can use the crystalline Torelli theorem due to Ogus \cite{Og83} instead of the Hodge-theoretic Torelli theorem for complex K3 surfaces. For K3-covers of finite height, we use characteristic-$0$-liftings as in \cite{LM18} since there is no Torelli-type theorem available. The proof of Proposition \ref{finitefield} will be given at the end of Section \ref{sec3}. In Section \ref{sec4}, we prove Theorem \ref{nonfinite} by explicitly constructing a desired surface from a certain Enriques surface in characteristic $p > 3$.

\medskip\noindent{\bf Conventions.} Throughout this paper, we fix an algebraically closed field $k$ of odd characteristic $p>2$. For a variety $V$ defined over a field $\mathbf{K}$, we denote the group of biregular (resp. birational) automorphisms of $V$ (over $\mathbf{K}$) by $\Aut(V/\mathbf{K})$ (resp. $\Bir(V/\mathbf{K})$). We will omit $\mathbf{K}$ and write $\Aut(V)$ (resp. $\Bir(V)$) if there is no ambiguity. For closed subsets $W_1, W_2, \dots, W_n\subset V$, we denote the \textit{decomposition group} and the \textit{inertia group}, respectively, by
\[ \mathrm{Dec}(V, W_1, \dots,W_n) := \{f \in \Aut(V)\,|\, f(W_i) = W_i\ \mathrm{for\ all}\ i\}, \]
\[ \mathrm{Ine}(V, W_1, \dots,W_n) := \{f \in \mathrm{Dec}(V, W_1, \dots,W_n)\,|\,f|_{W_i} = \mathrm{id}_{W_i}\ \mathrm{for\ all}\ i\}. \]

\medskip\noindent{\bf Acknowledgements.} I thank Professor Keiji Oguiso for valuable suggestions. I also thank Professor Igor Dolgachev for his comments. Special thanks to the referee for pointing out some gaps with very constructive remarks. 
I am grateful to the University of Tokyo Special Scholarship for International Students.

\section{Preliminaries}\label{sec2}

Let $X$ be an Enriques surface over $k$ and let $\widetilde{X}$ be its universal cover. It is well-known that $\widetilde{X}$ is a K3 surface and we call it the \textit{K3-cover} of the Enriques surface $X$. Moreover, we have a fixed-point free involution $\theta$, called an \textit{Enriques involution}, such that $\widetilde{X}/\langle\theta\rangle = X$.

For a K3 surface $\widetilde{X}$ over $k$, consider the functor that associates to every local Artinian $k$-algebra $A$ with residue field $k$ the Abelian group
\[ \Phi^2_{\widetilde{X}/k}: A \to \Ker (H^2_{\acute{e}t}(\widetilde{X}\times_k \Spec\,A, \mathbb{G}_m) \to H^2_{\acute{e}t}(\widetilde{X}, \mathbb{G}_m)). \]
This functor is pro-representable by a $1$-dimensional formal group law $\widehat{\Br}(\widetilde{X})$, which is called the \textit{formal Brauer group} of $\widetilde{X}$. The height $h$ of the formal Brauer group $\widehat{\Br}(\widetilde{X})$ satisfies $1 \leq h \leq 10$ or $h = \infty$ (see \cite{Li16} for a survey).

\begin{definition} A K3 surface is called supersingular, if its formal Brauer group has infinite height.
\end{definition}

A K3 surfaces with maximal possible Picard number $22$ is supersingular. Conversely, using the solution to the Tate conjecture for supersingular K3 surfaces, one deduces that a supersingular K3 surface always has Picard number $22$. The discriminant of the intersection form on the N\'{e}ron--Severi group $\NS(\widetilde{X})$ of a supersingular K3 surface $\widetilde{X}$ is equal to $\mathrm{disc}\,\NS(\widetilde{X}) = -p^{2\sigma}$ for some integer $1 \leq \sigma \leq 10$. The integer $\sigma$ is called the \textit{Artin invariant} of the supersingular K3 surface. Ogus \cite{Og83} established a Torelli theorem for supersingular K3 surfaces with marked N\'{e}ron--Severi lattices in terms of crystalline cohomology.

We begin with the notion of the period for a marked supersingular K3 surface $(\tx, \eta)$. Let $\tx$ be a supersingular K3 surface with Artin invariant $\sigma$ and $\eta: N\to \NS(\tx)$ be an isometry, where $N$ is a fixed (supersingular) K3 lattice. The composite of $\eta$ and the Chern class map $c_{\mathrm{dR}}: \NS(\tx)\to H^2_{\mathrm{dR}}(\tx/k)$ gives a map $\bar{\eta}: N\otimes k\to H^2_{\mathrm{dR}}(\tx/k)$. It can be shown that $\Ker(\bar{\eta})\subseteq N_0\otimes k$, where $N_0:=pN^{\ast}/pN$ is a $2\sigma$-dimensional $\mathbb{F}_p$-vector space. Then we define the \textit{period}
\[ \mathfrak{K}:= \mathfrak{K}_{(\tx, \eta)}:= (\mathrm{id}_{N_0}\otimes F_k)^{-1}(\Ker(\bar{\eta}))\subset N_0\otimes k, \]
where $F_k: k\to k$ is the Frobenius.

For any element $b$ in a fixed K3 lattice $N$ satisfying $(b, b) = -2$, let $r_b$ be the reflection of $N$ given by $x\mapsto x + (x, b)\cdot b$, and let $W_N$ be the subgroup of isometries of $N$ generated by all of these $r_b$. We call $W_N$ the \textit{Weyl group} of the K3 lattice $N$.

\begin{theorem}[{crystalline Torelli theorem, \cite[Page 371]{Og83}}]\label{torelli} Let $(\tx, \eta)$ be a marked supersingular K3 surface and let $G_{\mathfrak{K}}$ be the subgroup of isometries of $\NS(\tx)$ preserving both the positive cone $\Pos(\tx)$ and the period $\eta(\mathfrak{K})$. Then the automorphism group $\Aut(\tx)$ equals the subgroup of $G_{\mathfrak{K}}$ consisting of elements that preserve the ample cone $\Amp(\tx)$. The Weyl group $W_{\widetilde{X}}:= W_{NS(\widetilde{X})}$ is a normal subgroup of $G_{\mathfrak{K}}$ and $G_{\mathfrak{K}} = W_{\tx} \rtimes \Aut(\tx)$.
\end{theorem}

\begin{proposition}[{\cite[Proposition 5.2]{LM18}}]\label{lm5.2} Keep the same notations as above. Then $G_{\mathfrak{K}}$ is a subgroup of finite index of the group of isometries of $N$ preserving the positive cone $\Pos(\tx)$.
\end{proposition}

The following result is available for all K3 surfaces (see e.g., \cite[Proposition 1.10]{Og83}).

\begin{proposition}\label{action} Let $\widetilde{X}$ be a K3 surface over $k$. Then the nef cone $\Nef(\tx)$ is a fundamental domain for the action of $W_{\widetilde{X}}$ on the positive cone $\Pos(\tx)$.
\end{proposition}

We next consider the lifting of a K3 surface of finite height to characteristic $0$. Let $R$ be a fixed  complete discrete valuation ring with residue field $k$ and fraction field $K$, e.g, $R = W(k)$ the Witt ring of $k$. We denote by $\kba$ a fixed algebraic closure of $K$.

\begin{theorem}\label{general} Let $\pi: \tfx \to \Spec\,R$ be a smooth proper relative K3 surface. If the specialization map $\mathrm{sp}: \NS(\tfx_{\kba}) \to \NS(\tfx_k)$ is an isomorphism, then it induces an isomorphism of nef cones $\Nef(\tfx_{\overline{K}}) \to \Nef(\tfx_k)$ respecting the ample cones. Moreover, there is an injective group homomorphism $\phi: \Aut(\tfx_{\kba}) \to \Aut(\tfx_k)$ such that $\spp$ is $\phi$-equivariant with respect to the natural pullback actions.
\end{theorem}

\begin{proof} This follows from \cite[Corollary 2.4 and Theorem 2.1]{LM18} and the fact that $H^0(\tfx_k, T_{\tfx_k}) = 0$ for the K3 surface $\tfx_k$.
\end{proof}

For a K3 surface $\widetilde{X}$ over $k$, an automorphism $\alpha \in \Aut(\widetilde{X})$ is called \textit{tame} if $\alpha$ is of finite order and the order of $\alpha$ is not divisible by the base characteristic $p$. If $\widetilde{X}$ is of finite height, we call the orthogonal complement of the image of the Chern class map $c_{\mathrm{cris}}: \NS(\widetilde{X})\otimes W \hookrightarrow H^2_{\mathrm{cris}}(\widetilde{X}/W)$ the \textit{crystalline transcendental lattice} of $\widetilde{X}$ and denote it by $T_{\mathrm{cris}}(\widetilde{X})$. We denote the representation of
$\Aut(\widetilde{X})$ on $T_{\mathrm{cris}}(\widetilde{X})$ by $\chi_{\mathrm{cris},\widetilde{X}}: \Aut(\widetilde{X}) \to O(T_{\mathrm{cris}}(\widetilde{X}))$. An automorphism $\alpha \in \Aut(\widetilde{X})$ is called \textit{weakly tame} if the order of $\chi_{\mathrm{cris},\widetilde{X}}(\alpha)$ is not divisible by $p$. A tame automorphism is clearly weakly tame.

We say that an automorphism $\alpha \in \Aut(\widetilde{X})$ is \textit{liftable over the Witt ring} $W=W(k)$ if there is a proper scheme lifting $\widetilde{\mathfrak{X}} \to \Spec\,W$ of $\widetilde{X}$ and a $W$-automorphism $\mathfrak{a}: \widetilde{\mathfrak{X}} \to \widetilde{\mathfrak{X}}$ such that the restriction $\mathfrak{a}|_{\widetilde{X}}$ of $\mathfrak{a}$ to the special fiber is equal to $\alpha$. The pair $(\widetilde{\mathfrak{X}}, \mathfrak{a})$ is called a \textit{lifting} of the automorphism $\alpha$ of $\tx$. We have the following result due to Jang \cite[Theorem 3.2]{Ja17}.

\begin{proposition}\label{liftable} Let $\widetilde{X}$ be a K3 surface over $k$ of finite height and let $\alpha \in \Aut(\widetilde{X})$ be a weakly tame automorphism, then there is a N\'{e}ron--Severi group preserving lifting of $\alpha$. 
\end{proposition}

\begin{corollary}\label{liftfinite} Let $\widetilde{X}$ be a K3 surface over $k$ of finite height and let $\theta \in \Aut(\widetilde{X})$ be an Enriques involution, then there is a N\'{e}ron--Severi group preserving lifting $(\widetilde{\mathfrak{X}}, \Theta)$ of $\theta$ satisfying that  $\Theta$ is fixed-point free. 
\end{corollary}

\begin{proof} Recall that the base characteristic is $p > 2$ by assumption, so an Enriques involution is tame. By Proposition \ref{liftable}, it remains to show that $\Theta$ is fixed-point free. This is because the set of $\Theta$-fixed points is closed in $\widetilde{\mathfrak{X}}/W$ and its intersection with the special fibre $\tx$ is empty.
\end{proof}

Finally, we note the following two purely group-theoretic facts (see e.g., \cite[Page 25 and 181]{Su82}) which will be used later.

\begin{lemma}\label{useful} Let $G$ be a group and let $H, K \subseteq G$ be two subgroups. If $H\subseteq G$ is of finite index, then $H\cap K\subseteq K$ is also of finite index.
\end{lemma}

\begin{proposition}\label{group} Let $G$ be a group and let $H \subseteq G$ be a subgroup. Assume that $H$ is of finite index. Then the group $H$ is finitely generated if and only if $G$ is finitely generated.
\end{proposition}

\section{Finiteness and the cone theorem}\label{sec3}

We turn to the relation between an Enriques surface and its K3-cover. Let us first fix the following notations. Recall that $X$ is an Enriques surface over $k$ and $\widetilde{X}$ is its K3-cover with the Enriques involution $\theta$.

\begin{notation} {\rm (1)} Let $\NS(\tx)$ be the N\'{e}ron--Severi lattice of $\tx$ and
\[ L := \{x\in \NS(\tx)\,|\,\theta^{\ast}x=x\}\cong \NS(X)/\mathrm{torsion}. \]

{\rm (2)} Let $\Pos(\tx)$ be the positive cone of $\tx$, i.e., the connected component of the space $\{x\in\NS(\tx)\otimes \mathbb{R}\,|\, (x, x) > 0\}$ containing the ample classes, and let $\Pos^{+}(\tx)$ be the union of $\Pos(\tx)$ and all $\mathbb{Q}$-rational rays in the boundary $\partial\,\Pos(\tx)$ of $\Pos(\tx)$ in $\NS(\tx)\otimes \mathbb{R}$.

\medskip {\rm (3)} Similarly define $\Pos(X)$ and $\Pos^{+}(X)$ for the Enriques surface $X$. Then $\Pos(X)$ is the $\theta$-invariant part of $\Pos(\tx)$, and $\CCp:= \Pos^{+}(X)$ is the $\theta$-invariant part of $\Pos^{+}(\tx)$.

\medskip {\rm (4)} Since any effective nef divisor on a Calabi--Yau surface is semi-ample, we find that $\Nef^{e}(\tx):= \Nef(\tx)\cap \mathrm{Eff}(\tx) = \Nef(\tx)\cap\Pos^{+}(\tx)$ and that $\AAA:= \Nef^{e}(X) = \Nef(X)\cap\Pos^{+}(X)$ is the $\theta$-invariant part of $\Nef^{e}(\tx)$.

The interior $\AAc$ of $\AAA$ consists of $\theta$-invariant ample classes of $\tx$ and is non-empty. Moreover, $\AAc$ is the $\theta$-invariant part of $\Nef^{e}(\tx)^{\circ}$.

\medskip {\rm (5)} Let $O(L)^{+}$ be the orthogonal group of the lattice $L$ preserving $\CCp = \Pos^{+}(X)$, and let $\Aut(X)$ be the automorphism group of $X$ which is isomorphic to $\{f\in\Aut(\tx)\,|\,f\circ\theta = \theta\circ f\}/\langle\theta\rangle$. There is a natural map $\Aut(X)\to O(L)^{+}$ with finite kernel. Let $\Aut(X)^{\ast}$ be the image of $\Aut(X)$ in $O(L)^{+}$. Then each element in $\Aut(X)^{\ast}$ preserves $\AAA$.

\medskip {\rm (6)} For a supersingular K3-cover $\tx$ with period $\mathfrak{K}$, we further denote by $O(L)^{++}$ the subgroup of $O(L)^{+}$ consisting of the elements of the form $\tau|_L$, where $\tau$ is an isometry of $\NS(\tx)$ such that $\tau\circ\theta^{\ast}=\theta^{\ast}\circ\tau$ and $\tau(\mathfrak{K})=\mathfrak{K}$. Note that such a $\tau$ always preserves $L$.
\end{notation}

\begin{lemma}\label{finite} The inclusion $O(L)^{++}\subseteq O(L)^{+}$ is of finite index for a supersingular K3 surface $\tx$.
\end{lemma}

\begin{proof} The assertion follows from the fact that $O(L)^{++} = O(L)^{+}\cap G_{\mathfrak{K}}$, Proposition \ref{lm5.2}, and Lemma \ref{useful}.
\end{proof}

Let $\Nod(\tx) := \{ [C]\in \NS(\tx)\, |\, \PP^1 \cong C \subset \tx \}$ be the set of nodal classes of $\tx$. Set $\NN:= \{ b \in \Nod(\tx) \ \mathrm{satisfying} \ (b, \theta^{\ast}(b)) = 0 \}$. For each $b\in\NN$, we define                                      \[ R_b := r_b \circ r_{\theta^{\ast}(b)}, \]
where $r_b$ is the reflection associated with the nodal class $b$.

\begin{lemma}\label{obv} If $b \in \NN$, then
\[ R_b(x) = x + (x, b)\cdot b + (x, \theta^{\ast}(b))\cdot \theta^{\ast}(b) \]
for each $x\in \NS(\tx)$. In particular, $R_b = r_b\circ r_{\theta^{\ast}(b)} = r_{\theta^{\ast}(b)}\circ r_b$ and $R_b^2 = \mathrm{id}$.
\end{lemma}

\begin{proof} By a direct computation and the assumption that $(b, \theta^{\ast}(b)) = 0$,
\[ R_b(x) = r_b\circ r_{\theta^{\ast}(b)}(x) = r_b(x + (x, \theta^{\ast}(b))\cdot \theta^{\ast}(b)) \]
\[ = x + (x, b)\cdot b + (x, \theta^{\ast}(b))\cdot \theta^{\ast}(b) = r_{\theta^{\ast}(b)}\circ r_b(x). \]
Since $r_b^2 = r_{\theta^{\ast}(b)}^2 = \mathrm{id}$, we know that $R_b^2 = \mathrm{id}$.
\end{proof}

\begin{definition} The \textit{equivariant Weyl group} $W$ for the pair $(\tx, \theta)$ or for the Enriques surface $X$, is defined by $W:= \langle R_b\,|\,b\in\NN \rangle$, which is a subgroup of $O(\NS(\widetilde{X}))$.
\end{definition}

\begin{lemma}\label{basic} {\rm (1)} $W\subseteq O(L)^{+}$. More precisely, for any $\sigma\in W$, the restriction $\sigma|_L$ is in $O(L)^{+}$ and if $\sigma|_L = \mathrm{id}$, then $\sigma = \mathrm{id}$. If in addition, $\tx$ is supersingular, then $W\subseteq O(L)^{++}$.

\medskip {\rm (2)} Let $\wnod$ be the subgroup of $O(L)^{+}$ generated by reflections associated with classes of nodal curves (i.e., $(-2)$-curves) on $X$. Then $W \cong \wnod$.

\medskip {\rm (3)} The cone $\AAA$ is a fundamental domain for the action of $W$ on $\CCp$, i.e., $W\cdot\AAA = \CCp$ and $\sigma(\AAc)\cap\AAc \neq \varnothing$ if and only if $\sigma = \mathrm{id}$.

\end{lemma}

\begin{proof} By Lemma 3.3 and a direct calculation, we see that $\theta^{\ast}\circ R_b =R_b\circ\theta^{\ast}$ and hence $R_b(L) = L$. Since $R_b$ preserves $\Pos^{+}(\tx)$, it also preserves $\CCp$ so that $R_b\in O(L)^{+}$. If $\tx$ is supersingular, by \cite[Corollary in Page 371]{Og83}, $r_b$ preserves the period $\mathfrak{K}$, hence so does $R_b$. Therefore, $R_b\in O(L)^{++}$.

Assume that $\sigma|_L =\mathrm{id}$ for some $\sigma\in W$. Then $\sigma(\AAc)\cap\AAc \neq \varnothing$, and in particular, $\sigma(\Nef^{e}(\tx)^{\circ})\cap\Nef^{e}(\tx)^{\circ} \neq \varnothing$. This implies $\sigma = \mathrm{id}$, since by Proposition \ref{action}, $\Nef^{e}(\tx)$ is a fundamental domain for the action of $W_{\tx}:=\langle r_b\,|\, b\in\Nod(\tx)\rangle$ on $\Pos^{+}(\tx)$. This proves Claim $\mathrm{(1)}$.

Let us show Claim $\mathrm{(2)}$. Clearly the class of a nodal curve on $X$ belongs to $\NN$. On the other hand, each element $b \in \NN$ defines a nodal curve on $X$. Thus, there is an isomorphism $W \cong \wnod$.

For Claim $\mathrm{(3)}$, it remains to check the equality $W\cdot\AAA = \CCp$. 
Let us first note that $(x, b) = (\theta^{\ast}(x), \theta^{\ast}(b)) = (x, \theta^{\ast}(b))$ for all $x\in L_{\mathbb{R}}:= L\otimes \mathbb{R}$ and $b\in \Nod(\tx)$, so for $b\in \NN$, we have
\[ R_b(x) = x + (x,b)\cdot (b + \theta^{\ast}(b)). \]
This formula shows that $R_b|_{L_{\mathbb{R}}}$ is a reflection with respect to the hyperplane defined by $(-, b) = 0$. Recall that the subgroup of the orthogonal group $O(L_{\mathbb{R}})$ preserving the interior $(\mathcal{C}^{+})^{\circ}$ of $\mathcal{C}^{+}$ makes the quotient space $(\mathcal{C}^{+})^{\circ}/\mathbb{R}_{>0}$ a Lobachevsky space. It is a general fact that the fundamental
domain of a discrete group generated by reflections in hyperplanes of a Lobachevsky space, is the closure of a chamber (see \cite{Vi71}). Hence the cone $\AAA = \{ x \in \mathcal{C}^{+} \,|\,(x, b) \geq 0\ \mathrm{for\ all}\ b \in \NN \}$ is a fundamental domain for the action $W$ on $\mathcal{C}^{+}$.
\end{proof}

\begin{theorem}\label{equivtorelli} Let $G$ be the subgroup of $O(L)^{+}$ generated by $W$ and $\Aut(X)^{\ast}$. Then the following assertions hold.

\medskip {\rm (1)} $W$ is a normal subgroup of $G$ and $G = W\rtimes \Aut(X)^{\ast}$.

\medskip {\rm (2)} If the K3-cover $\tx$ is supersingular, then precisely $G = O(L)^{++}$.

\medskip {\rm (3)} The inclusion $G\subset O(L)^{+}$ is of finite index.
\end{theorem}

\begin{proof} (1) We first check that for each $b\in \NN$ and $\sigma\in \Aut^{\ast}(X)$, there exists an element $b^{\prime}\in \NN$ such that $\sigma^{-1}\circ R_b\circ \sigma= R_{b^{\prime}}$. Choosing $f\in \Aut(X)$ such that $f^{\ast} = \sigma$, we obtain that $\sigma^{-1}\circ R_b\circ\sigma = (f^{-1})^{\ast}\circ R_b\circ f^{\ast}$. 
Therefore,
\[ \sigma^{-1}\circ R_b\circ\sigma(x) = x + (x, (f^{-1})^{\ast}(b)) \cdot(f^{-1})^{\ast}(b) + (x, \theta^{\ast}\circ(f^{-1})^{\ast}(b)) \cdot(\theta^{\ast}\circ(f^{-1})^{\ast}(b)). \]
Note that $f^{\ast}\circ\theta^{\ast} =\theta^{\ast}\circ f^{\ast}$, so we can take $b^{\prime}:= (f^{-1})^{\ast}(b) \in \NN$. This shows that $W$ is a normal subgroup of $G$. We next check the uniqueness of the decomposition. Assume that $r\circ\tau = r^{\prime}\circ\tau^{\prime}$ for $r, r^{\prime}\in W$ and $\tau, \tau^{\prime}\in \Aut(X)^{\ast}$. Since $\tau\circ(\tau^{\prime})^{-1}(\AAA) = \AAA$, we see that $r^{-1}\circ r^{\prime}(\AAA)= \AAA$. Therefore, $r^{-1}\circ r^{\prime} = \mathrm{id}$ by Lemma \ref{basic} (3).

\medskip (2) It remains to show that $O(L)^{++} = W\cdot \Aut(X)^{\ast}$. Let $\sigma$ be an element of $O(L)^{++}$ and take $x\in\AAc$. By Lemma \ref{basic} (3), we get an element $r\in W$ such that $r^{-1}\circ\sigma(x)\in\AAc$. Note that $r^{-1}\circ\sigma\in O(L)^{++}$ by Lemma \ref{basic} (1), so there exists an isometry $\rho$ of $\NS(\tx)$ preserving the period $\mathfrak{K}$, $\rho\circ\theta^{\ast} = \theta^{\ast}\circ\rho$, and $\rho|_L = r^{-1}\circ\sigma$. Moreover, $\rho$ preserves the ample cone $\Amp(\tx)$ of $\tx$, as $\rho(x) = r^{-1}\circ\sigma(x)\in \rho(\Nef^{e}(\tx)^{\circ})\cap\Nef^{e}(\tx)^{\circ}$. Hence by Theorem \ref{torelli}, there exists $f\in\Aut(\tx)$ such that $f^{\ast} = \rho$. Furthermore, $f\circ\theta = \theta\circ f$ again by Theorem \ref{torelli}, because $f^{\ast}\circ\theta^{\ast} = \theta^{\ast}\circ f^{\ast}$. Hence $f^{\ast}\in \Aut(X)^{\ast}$. As $f^{\ast} = r^{-1}\circ\sigma$, we get $O(L)^{++} = W\cdot \Aut(X)^{\ast}$.

\medskip (3) If the K3-cover $\tx$ is supersingular, then it follows from the assertion (2) and Lemma \ref{finite}. In what follows, we assume that $\tx$ is of finite height. By Corollary \ref{liftfinite}, we have a smooth projective relative K3 surface $\mathfrak{X} \to \Spec\,W$ with the specialization map being isomorphic and a fixed-point free automorphism $\Theta$ of $\widetilde{\mathfrak{X}}$ satisfying that $\Theta|_{\tx} = \theta$. If we take the quotient $\mathfrak{X}:= \widetilde{\mathfrak{X}}/\Theta$, then clearly $\mathfrak{X} \to \Spec\,W$ is a relative Enriques surface with the special fiber $X$. We set $X_0 := \mathfrak{X}_{\overline{K}}$ and $\tx_0 := \mathfrak{\widetilde{X}}_{\overline{K}}$ for simplicity. Then $\tx_0$ is the K3-cover of $X_0$ with the Enriques involution $\theta_0$ induced by $\theta$. By Theorem \ref{general}, there is an injective homomorphism $\phi: \Aut(\tx_0) \to \Aut(\tx)$ such that the specialization map $\spp$ is $\phi$-equivariant. We can identify $\theta_0$ and $\theta$ via $\phi$, so that $\Aut(X_0)\subset \Aut(X)$. This is because of the fact that $\Aut(X) = \{f\in\Aut(\tx)\,|\,f\circ\theta = \theta\circ f\}/\langle\theta\rangle$. Moreover, $\NS(X_0)/\mathrm{torsion}\cong \NS(X)/\mathrm{torsion}\cong L$, so $\Aut(X_0)^{\ast}\subset \Aut(X)^{\ast}\subset O(L)^{+}$. Let $W_0$ be the equivariant Weyl group for $X_0$. Then by the construction, it follows that $W_0 = W$. Therefore, we find that $W_0\rtimes \Aut(X_0)^{\ast}\subset W\rtimes\Aut(X)^{\ast}\subset O(L)^{+}$. By the main theorem of \cite{Do84}, the inclusion $W_0\rtimes \Aut(X_0)^{\ast}\subset O(L)^{+}$ is of finite index, since $X_0$ is defined over an algebraically closed field of characteristic $0$. Thus, $G = W\rtimes\Aut(X)^{\ast}\subset O(L)^{+}$ is of finite index.
\end{proof}


We digress to the following finiteness result for the bicanonical representation of the automorphism group of an Enriques surface which will be used in Section \ref{sec4}.  The corresponding result over complex number field is well-known (see e.g., \cite[Section 14]{Ue75}). The proof is inspired by \cite[Proposition 3.5]{Ja16}.

\begin{proposition}\label{represent} Let $X$ be an Enriques surface over $k$ whose K3-cover $\tx$ is of finite height. Let $\chi: \Aut(X) \to \GL(H^0(X, \omega_X^{\otimes 2})) \cong k^{\times}$ be the bicanonical representation of $\Aut(X)$. Then the image of $\chi$ is a finite group (and hence a finite cyclic group).
\end{proposition}

\begin{proof} As in the proof of Theorem \ref{equivtorelli} (3), $X$ has a lift $X_0$ to an algebraically closed field of characteristic $0$ such that the specialization map is isomorphic and equivariant with respect to the injective map $\phi: \Aut(X_0) \to \Aut(X)$. Then the image of $\Aut(X_0) \to \GL(H^0(X_0, \omega_{X_0}^{\otimes 2}))$ is finite. Therefore, the image of $\Aut(X_0) \hookrightarrow \Aut(X) \to \GL(H^0(X, \omega_X^{\otimes 2}))$ is also finite. On the other hand, the proof of Theorem \ref{equivtorelli} (3) shows that $\Aut(X_0)^{\ast}\subset\Aut(X)^{\ast}$ is of finite index. Since the kernel of the map $\Aut(X_0) \hookrightarrow \Aut(X) \to O(L)^{+}$ is finite, it follows that $\Aut(X_0)\subset\Aut(X)$ is also of finite index. We then obtain that the image of $\chi$ is finite.
\end{proof}

We return to complete the proofs of Theorem \ref{finiteness} and \ref{cone}.

\begin{proof}[Proof of Theorem \ref{finiteness}] As up to finite kernel, $\Aut(X)$ is isomorphic to a quotient of arithmetic group by Theorem \ref{equivtorelli} (3), so it is finitely generated by a result of Borel-Harish-Chandra \cite{BH62} and Proposition \ref{group}.
\end{proof}

\begin{corollary}[{\textit{cf.} \cite[Corollaries 3.4 and 3.5]{Do84}}]\label{do3.45} $\Aut(X)$ is finite if and only if $\wnod$ is of finite index in $O(L)^{+}$. In particular, $\Aut(X)$ is infinite if $X$ does not contain nonsingular rational curves.
\end{corollary}

\begin{proof} This follows from Theorem \ref{equivtorelli} and Lemma \ref{basic} (2).
\end{proof}

\begin{proof}[Proof of Theorem \ref{cone}] Note that $G = W\rtimes \Aut(X)^{\ast}$ is of finite index in the arithmetic group $O(L)^{+}$ of the self-dual homogeneous cone $\mathcal{C}:= \Pos(X)$. Then by a general result (\cite[Corollary II.4.9]{AMRT10}), there exists a rational polyhedral fundamental domain $\Pi$ for the action $G$ on $\mathcal{C}^{+}$. We may translate $\Pi$ by a suitable element of $W$ and assume that $\Pi^{\circ}\cap\AAc\neq\varnothing$.

\begin{claim} \textit{This $\Pi$ satisfies that $\Pi\subseteq \mathcal{A}$.}
\end{claim}

\begin{proof}[Proof of Claim] Otherwise, there would be an element $b\in\Nod(\tx)$ such that the wall $H_b:=\{x\in L\otimes \mathbb{R}\,|\, (x, b) = 0\}$ of $\mathcal{A}$ satisfies $\Pi^{\circ}\cap H_b\neq\varnothing$. However, since $R_b(x) = x$ for $x\in \Pi^{\circ}\cap H_b$, it follows that $R_b(\Pi^{\circ}) \cap \Pi^{\circ} \neq\varnothing$, which is a contraction.
\end{proof}

Now combining the facts that $\Aut(X)^{\ast} \cong G/W$, $\Aut(X)^{\ast}\cdot \mathcal{A} = \mathcal{A}$, and that $\mathcal{A}$ is a fundamental domain for the action of $W$ on $\CCp$ by Lemma \ref{basic} (3), we conclude that this $\Pi$ gives a desired rational polynomial fundamental domain. This completes the proof.
\end{proof}


\begin{corollary}\label{smoothrational} On an Enriques surface $X$ over $k$ there are only finitely many smooth rational curves modulo automorphisms of $X$.
\end{corollary}

\begin{proof} This follows immediately from the geometric meaning of the cone theorem. Indeed, each contraction of $X$ to a projective surface is given by some semi-ample line bundle on $X$. The class of such a line bundle lies in the nef effective cone, and two semi-ample line bundles in the interior of the same face of the rational polyhedral fundamental domain determine the same contraction of $X$. On the other hand, one can contract each smooth rational curve on $X$ by some contraction of $X$.
\end{proof}

\begin{corollary}\label{ellipticpencil} On an Enriques surface $X$ over $k$ there are only finitely many elliptic fibrations modulo automorphisms of $X$.
\end{corollary}

\begin{proof} We follow the argument in \cite[Theorem 6.7]{Na85}. First note that, by our assumption, $\mathrm{char}(k) = p \geq 3$, so quasi-elliptic fibrations would appear only when $p = 3$. But by \cite[Proposition 2.7]{La79}, there exist no quasi-elliptic fibrations on Enriques surfaces over a field of characteristic three.

By \cite[Proposition 6.6]{Na85}, we can identify the set of elliptic pencils on $X$ with the set
\[ \mathscr{B}:= \{\mathbb{R}\cdot e \,|\, e \in L\ \mathrm{with}\ e \ \mathrm{nef} \ \mathrm{and} \ (e, e) = 0\}. \]
Therefore, it is enough to show that there are only finitely many elements in $\mathscr{B}$ modulo $\Aut(X)^{\ast}$. By the Cone Theorem \ref{cone}, $\Aut(X)^{\ast}$ acts on the effective nef cone $\Nef^e(X)$ with a rational polyhedral fundamental domain $\Pi$, which contains only finitely many rational isotropic vectors modulo constant multiple. This implies the finiteness of the set $\mathscr{B}$ modulo $\Aut(X)^{\ast}$.
\end{proof}

We finally turn to the proof of Proposition \ref{finitefield} following the argument in \cite[Theorem 1.1 (1)]{Og19}.

\begin{lemma}\label{basechange} Let $V$ be a projective variety defined over an algebraically closed field $\mathbf{K}$. Assume that $\Aut(V/\mathbf{K})$ is discrete. Then $\Aut(V_{\mathbf{L}}/\mathbf{L}) = \Aut(V/\mathbf{K})$ for any field extension $\mathbf{L}/\mathbf{K}$.
\end{lemma}

\begin{proof} See \cite[Lemma 2.2]{Og19}.
\end{proof}

\begin{lemma}\label{minimal} Let $X$ and $Y$ be two birationally  equivalent smooth projective surfaces over $\overline{\mathbb{F}}_p$. Assume $X$ is minimal and $\Aut(X/\overline{\mathbb{F}}_p)$ is finitely generated. Then $\Aut(Y/\overline{\mathbb{F}}_p)$ is also finitely generated.
\end{lemma}

\begin{proof} This is implied in the proof of \cite[Theorem 1.1.(1)]{Og19}.
\end{proof}

\begin{proof}[Proof of Proposition \ref{finitefield}] It follows from Theorem \ref{finiteness}, Lemma \ref{basechange} and Lemma \ref{minimal}.
\end{proof}

\section{Non-finite generatedness}\label{sec4}

In this section, we construct an explicit example which is birational to an Enriques surface and whose automorphism group is not finitely generated. Indeed, we show that the same construction as in \cite{KO19} gives the desired example at least in characteristic $p > 3$, after recalling some basic facts on Kummer surfaces of product type and Mukai's construction \cite{Mu10}.

\medskip We begin with the fact that $t\in \mathbb{F}_p(t)\subset \mathbf{K}$ is transcendental over $\mathbb{F}_p$. Consider two elliptic curves $E$ and $F$ defined by the Weierstrass equations $y^2 = x(x - 1)(x - t)$, and $v^2 = u(u - 1)(u - s)$ ($s\in \mathbf{K}$ and $s\neq 0, 1$), respectively. The branch points of the quotient map $E \to E/\langle -1_E\rangle\cong \mathbb{P}^1$ given by $(x, y) \mapsto x$ are exactly the points $0$, $1$, $t$ and $\infty$ of $\mathbb{P}^1$. The same assertion holds for $F$ if $t$ is replaced by $s$.

We note that $E$ is not supersingular (see e.g., \cite[Page 200]{Mu08}). If we take a supersingular elliptic curve defined over $\mathbf{K}$ as $F$, then we deduce that $E$ and $F$ are not isogenous over $\mathbf{K}$ (see e.g., \cite[Page 137]{Mu08}). Thus, we make the following assumption throughout Section \ref{sec4}.

\begin{assumption}\label{assumption} The two elliptic curves $E$ and $F$ are not isogenous.
\end{assumption}

Let $\tx:= \mathrm{Km}(E \times F)$ be the Kummer surface accociated to the abelian surface $E \times F$. Since $E$ and $F$ are not isogenous, the Picard number $\rho(\tx)$ of $\tx$ is $18$ (see e.g., \cite{Sh75}), so that $\tx$ is non-supersingular. As in Figure \ref{fig}, $\tx$ contains $24$ smooth rational curves $E_i, F_i\,(1 \leq i \leq 4)$ and $C_{ij}\,(1 \leq i, j \leq 4)$, which form the so-called \textit{double Kummer pencil} on $\tx$. Let $\{a_i\}^4_{i=1}$ (resp. $\{b_i\}^4_{i=1}$) be the $2$-torsion points of $F$ (resp. $E$), then the $E_i$ (resp. $F_i$) arise from the elliptic curves $E \times \{a_i\}$ (resp. $\{b_i\} \times F$) on $E \times F$, while the $C_{ij}$ are the exceptional curves over the $A_1$-singular points of the quotient surface $E \times F/\langle -1_{E\times F}\rangle$. We will use the names of rational curves in Figure \ref{fig} throughout Section \ref{sec4}.

\begin{figure}[h]
	\centering
	\includegraphics[width=0.44\linewidth]{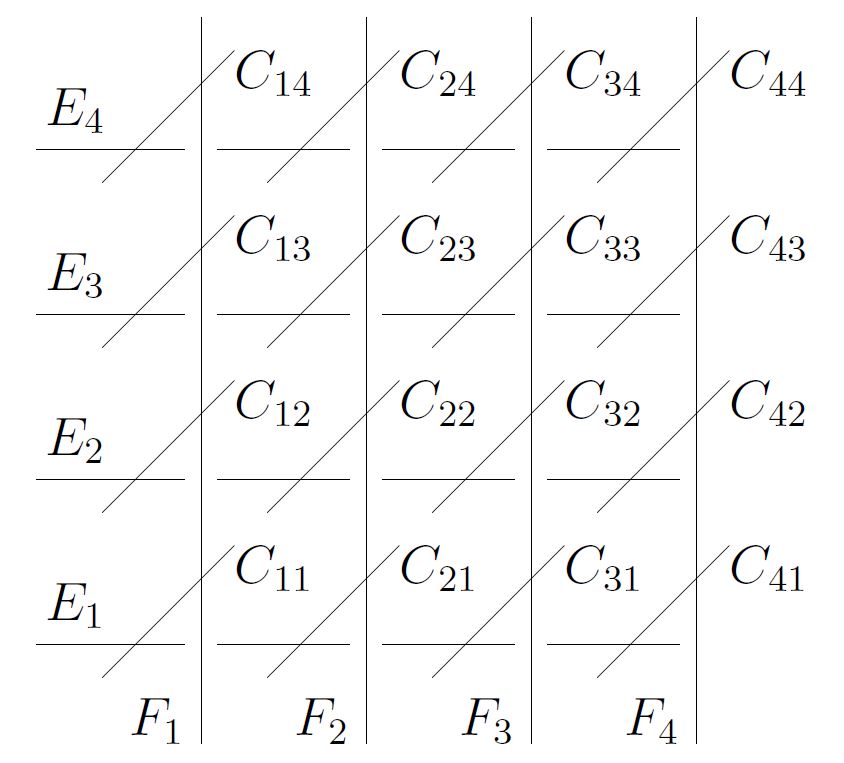}
	\caption{Curves $E_i$, $F_j$ and $C_{ij}$}
	\label{fig}
\end{figure}

Let us denote the unique point $E_j \cap C_{ij}$ (resp. $F_i \cap C_{ij}$) by $P_{ij}$ (resp. $P^{\prime}_{ij}$). Choose affine coordinates $x$ and $u$ of $E_j$ and $F_i$, respectively, such that
\[ P_{1j} = 1, \ P_{2j} = t, \ P_{3j} = \infty, \ P_{4j} = 0 \]
on $E_j$ with respect to the coordinate $x$, and
\[ P^{\prime}_{i1} = 1, \ P^{\prime}_{i2} = s, \ P^{\prime}_{i3} = \infty, \ P^{\prime}_{i4} = 0 \]
on $F_i$ with respect to the coordinate $u$.

\medskip We recall the construction of Mukai \cite{Mu10} for $\tx = \mathrm{Km}(E \times F)$. Although Mukai works over the complex numbers, we can check that the same construction is also available in our setting.

Consider $\varepsilon := [(1_E, -1_F)] \in \Aut(\tx)$, which is an anti-symplectic involution of $\tx$. Let $T := \tx/\langle\varepsilon\rangle$ be the quotient surface and let $q : \tx \to T$ be the quotient map. Then $T$ is a smooth projective surface which can also be obtained by blowing up $\mathbb{P}^1 \times \mathbb{P}^1$ at the $16$ points. We identify $\mathbb{P}^1 \times \mathbb{P}^1$ with a smooth quadric surface $Q$ in $\mathbb{P}^3$ via the Segre embedding $\mathbb{P}^1 \times \mathbb{P}^1 \subset \mathbb{P}^3$. Then there exists a Cremona involution of $\mathbb{P}^3$ preserving $Q$ (see \cite[Section 2]{Mu10}), so that it restricts to a birational automorphism $\tau^{\prime} \in \Bir(Q)$. Moreover, $\tau^{\prime}$ lifts to $\tau \in \Aut(T)$ by \cite[Lemma 9]{Mu10}, and $\tau$ further lifts to an anti-symplectic involution $\theta \in \Aut(\tx)$.

Let us set $X := \tx/\langle\theta\rangle$, and denote the quotient map by $\pi : \tx \to X$. Then by \cite[Proposition 2]{Mu10}, $X$ is an Enriques surface with a numerically trivial involution induced by $\varepsilon \in \Aut(\tx)$. We introduce the following notations for curves and points on the Enriques surface $X$:
\[ H_j := \pi(E_j),\ D_{ij} := \pi(C_{ij}),\ Q_{ij} := \pi(P_{ij}). \]
Via the isomorphism $\pi|_{E_j} : E_j \to H_j$, we also regard $x$ as an affine coordinate of $H_j$. Then $Q_{ij} \in H_j$ and
\[ x(Q_{1j}) = 1,\ x(Q_{2j}) = t,\ x(Q_{3j}) = \infty,\ x(Q_{4j}) = 0. \]
We have $\pi^{-1}(H_j) = E_j \cup F_j$ for all $1 \leq j \leq 4$ and
\[ \pi^{-1}(D_{ij}) = C_{ij} \cup C_{ji},\ \pi^{-1}(Q_{ij}) = \{P_{ij}, P_{ji}^{\prime}\} \]
for all $i \neq j$, while
\[ \pi^{-1}(D_{ii}) = C_{ii} \cup \theta(C_{ii}),\ \pi^{-1}(Q_{ii}) = \{P_{ii} \cup P_i\}, \]
again for all $1 \leq j \leq 4$, where $P_i$ is the unique intersection point of $\theta(C_{ii}) \cap F_i$.

\begin{construction} Let $\mu_1 : X_1 \to X$ be the blow-up at the point $Q_{32} \in H_2$, i.e., the blowing up at $\infty$ under the coordinate $x$ of $H_2$. Let $E_{\infty} := \mathbb{P}(T_{X, Q_{32}}) \cong \mathbb{P}^1$ be the exceptional divisor of $\mu_1$. We choose three mutually different points on $\mathbb{P}(T_{X, Q_{32}})$, say $Q_{32k}\,(1\leq k \leq 3)$. Let $\mu_2 : X_2 \to X_1$ be the blow-up of $X_1$ at the $3$ points $Q_{32k}$. Set $\mu := \mu_1 \circ \mu_2 : X_2 \to X_1 \to X$. We denote by $E_{32k}$ the exceptional curve over $Q_{32k}$ under $\mu_2$ and by $E^{\prime}_{\infty}$ the proper transform of $E_{\infty}$ under $\mu_2$.
\end{construction}

\begin{theorem}\label{example} $\Aut(X_2)$ is not finitely generated.
\end{theorem}

\begin{proof} First we note that, this implies Theorem \ref{nonfinite}. As argued in \cite[Proposition 3.2]{KO19} by using the finiteness of the bicanonical representation in odd characteristic (Proposition \ref{represent}) instead, we obtain that \[ \mathrm{Ine}(X, Q_{32}, T_{Q_{32}}):= \{f \in \mathrm{Dec}(X, Q_{32})\,|\,df|_{T_{Q_{32}}} = \mathrm{id}_{T_{Q_{32}}}\} \]
is a subgroup of $\Aut(X_2)$ of finite index. Thus, it is enough to show that $\mathrm{Ine}(X, Q_{32}, T_{Q_{32}})$ is not finitely generated by Proposition \ref{group}.

By \cite[Lemma 3.3]{KO19} (see also \cite[Proposition 3.6]{Og19}), the differential maps $df|_{T_{X,Q_{32}}}$ are simultaneously diagonalizable for all $f \in \mathrm{Dec}(X, Q_{32})$. Combined with the isomorphism for each point $\mathcal{Q} \in \mathbb{P}^1$,
\[ (\mathbf{K}, +) \cong \{f \in \mathrm{Ine}(\mathbb{P}^1, \mathcal{Q})\,|\,df|_{T_{\mathbb{P}^1, \mathcal{Q}}} = \mathrm{id}_{T_{\mathbb{P}^1, \mathcal{Q}}}\} =: \mathrm{Ine}(\mathbb{P}^1, \mathcal{Q}, T_{\mathbb{P}^1, \mathcal{Q}})\]
given by $c \mapsto (z \mapsto z + c)$, if we choose an affine coordinate $z$ of $\mathbb{P}^1$ such that $z(\mathcal{Q}) = \infty$, we eventually obtain a representation
\[ \rho: \mathrm{Ine}(X, Q_{32}, T_{Q_{32}}) \to \mathrm{Ine}(H_2, Q_{32}, T_{H_2, Q_{32}}) \cong (\mathbf{K}, +), \]
where, for the last isomorphism, we use the affine coordinate $x$ of $H_2$.

By Proposition \ref{ko3.4} below, the image $\mathrm{Im}\,\rho$ contains the additive subgroup $M$ generated by $\{t^{-2n}a\,|\,n \in  \mathbb{Z}_{\geq 0}\}$, which is not finitely generated as $a \neq 0$ and $t$ is transcendental over $\mathbb{F}_p$. This implies that the abelian group $\mathrm{Im}\,\rho$ itself is not finitely generated, and hence neither is $\mathrm{Ine}(X, Q_{32}, T_{Q_{32}})$.
\end{proof}

\begin{proposition}\label{ko3.4} There is an element $a \in \mathbf{K}^{\times}$ such that $t^{-2n}a \in \mathrm{Im}\,\rho$ for all positive integers $n$.
\end{proposition}

\begin{proof}[Sketch of Proof] We refer to \cite{KO19} for full details. As in the proof of \cite[Proposition 3.4]{KO19}, one constructs two genus one fibrations $\varphi_{M_i} : X \to \mathbb{P}^1\,(i =1, 2)$ on $X$ associated to divisors $M_1$ and $M_2$ on $X$ of Kodaira's type $I_8$ and $IV^{\ast}$, respectively, which are induced by elliptic fibrations $\Phi_i : \tx \to \mathbb{P}^1\,(i =1, 2)$ on the K3-cover $\tx$, respectively. By \cite[Table A]{Og89} (or \cite[Table 1]{KS08}) and the construction of the Enriques involution $\theta$, we obtain that $\varphi_{M_i}$ has no reducible fibers other than $M_i$ for $i = 1, 2$. Let us consider the proper smooth relatively minimal Jacobian fibrations $\varphi_i : R_i \to \mathbb{P}^1$ of $\varphi_{M_i}\,(i = 1, 2)$, which are rational elliptic surfaces.

As argued in [KO19, Page 9] by using Proposition \ref{represent} instead again, we obtain that there exists an element $g \in \mathrm{Ine}(X, Q_{32}, T_{Q_{32}})$ such that $\rho(g) \neq 0$. By using \cite[Lemma 2.17]{Ma19} instead of \cite[Lemma 2.6]{Ko86}, and \cite[Section III.17]{Ne64} (or \cite[Section 6]{Ta75}) instead of \cite[Theorem 9.1]{Ko63}, the same argument as in [KO19, Page 10 and 11] leads to an element $f\in \Aut(X)$ such that $f^2 \in \mathrm{Dec}(X, Q_{32})\subset \mathrm{Dec}(X, H_2)$ and $f^2(x) = t^2x$ on $H_2$, where $x$ denotes the coordinate on  $H_2$. The assertion then follows clearly.
\end{proof}

\end{document}